\documentclass[12pt]{amsart}
\usepackage{amssymb,amsmath,amsthm, hyperref, verbatim}
\newtheorem*{thm*}{Theorem}
\newtheorem{thm}{Theorem}[section]
\newtheorem*{lemma}{Lemma}

\theoremstyle{definition}

\newtheorem*{dfn}{Definition}

\newtheorem{corollary}[thm]{Corollary}

\begin{document}

\title[Benford's Law for Coefficients of Modular Forms and Partition Functions]
{Benford's Law for Coefficients of Modular Forms and Partition Functions}
\author{Theresa C. Anderson, Larry Rolen, and Ruth Stoehr}
\address{Department of Mathematics, University of Wisconsin, Madison, Wisconsin 53706 }
\email{tcanderon2@wisc.edu} \email{lrolen@wisc.edu} \email{rstoehr@emory.edu}

\thanks {The authors are grateful for the support of the NSF through the American Recovery and Reinvestment Act:NSF Research Training Grant(PI: Ono), and would like to thank the University of Wisconsin, Madison,
particularly for the extensive help and guidance from Ken Ono and Marie Jameson.  The authors would also like to thank the referee for helping to improving the quality of exposition. Additionally, we thank Amanda Folsom, Steve Miller, and Andrew Sills for their numerous useful comments.}
 
\begin{abstract}
Here we prove that Benford's law holds for coefficients of an infinite class of modular forms. Expanding the work of Bringmann and Ono on exact formulas for harmonic Maass forms, we derive the necessary asymptotics. This implies that the unrestricted partition function $p(n)$, as well as other natural partition functions, satisfy Benford's law.
\end{abstract}
\maketitle
\noindent

\section{Introduction and Statement of Results}

 It has long been observed that many naturally occurring statistics and arithmetic functions have some surprising properties. In 1881, astronomer Simon Newcomb noticed that the earlier pages in logarithm tables were more worn than later ones. Instead of the a priori estimate that the first digit will be 1 one-ninth of the time, he found in these instances that this frequency is approximately $30\%$ for the digit 1 and less than $5\%$ for the digit 9 \cite{SN}. This phenomenon, known as \emph{Benford's Law}, appears in a wide class of data including river lengths and population demographics. For a more detailed discussion on the history and previous work on the subject see \cite{B}, \cite{Hill1}, \cite{Hill2}.
\\\indent
Although this ``law" is well-known, it has only been proven to hold for a relatively small class of arithmetic functions. For example, Miller and Kontorovich prove Benford's law  for distributions of values of $L$-functions and the $3x+1$ problem in \cite{MillKon}. The purpose of this paper is to prove that the statistically observed frequencies dictated by Benford's law hold for an infinite class of sequences arising as the coefficients of modular forms, including the \emph{partition function} $p(n)$.
\\\indent We define a \emph{partition} of a non-negative integer $n\in \mathbb{N}$ to be any non-increasing sequence of positive integers which sum to $n$. The \emph{partition number} $p(n)$ is the number of partitions of $n$. For a sequence of positive integers $a(n)$, let
\begin{equation}
B(d,x,k;a(n)) = \frac{\# \text{ \{ $n$ } \leq \text{ $x$ : first digits of } a(n) \text{ in base $k$ are the string } d \} }{x}.
\end{equation}
\indent
We say $a(n)$ is \emph{Benford} if $\displaystyle\lim_{x \to \infty} B(d,x,k;a(n))\equiv(\log_{k}(d+1)-\log_{k}(d))$ (mod 1) for all integers $k\geq 2$.  We denote the space of Benford functions as $\mathcal{B}$. Note that for a function to belong to $\mathcal{B}$, Benford's Law must hold for any initial string of digits in any base. We include this level of generality to study frequencies such as $B(101,x,2;a(n))$ which counts the proportion of $n\leq x$ for which $a(n)$ begins with the string $d$=``101" in base 2. Since $101_2 =  5_{10}$ and $\log_{2}(6)-\log_{2}(5)\approx 0.263$ (mod 1), the predicted frequency for members of a Benford sequence to begin with $101_2$ is about $26.3\%$.

The following data for initial digits illustrates the plausibility of Benford's Law for $p(n)$ for $k=10$.
\begin{table}[ht]
\caption{$B(d,x,10;p(n))$} % title of Table
\scriptsize
\begin{tabular}{c c c c c c c c c c}  % centered columns (4 columns)
\hline\hline                        %inserts double horizontal lines
$x$ & $d$ = 1 & 2 & 3 &4&5&6&7&8&9\\ [0.5ex] % inserts table
%heading
\hline                    % inserts single horizontal line
$10^{2}$& 0.33&0.16& 0.14&0.09&0.07&
  0.06&0.07&0.05&0.03\\
$10^{3}$&0.305&0.177&0.127&0.094&0.076&
0.068&0.057&0.052&0.044\\
$10^{4}$& 0.302&0.177& 0.126& 0.096&
  0.078&0.067&0.057&0.051& 0.046\\
$\downarrow$ & $\downarrow$ & $\downarrow$ & $\downarrow$ & $\downarrow$ & 
$\downarrow$ & $\downarrow$ & $\downarrow$ & $\downarrow$ & $\downarrow$\\
$\infty$?&0.301&0.176&0.125&0.097& 0.079&0.067& 
0.057&0.051&0.046\\

\hline     %inserts single line
\end{tabular}
\label{table:nonlin}  % is used to refer this table in the text
\end{table}%can we shorten this a little in width -- a little far to the right
\\
We also have the following data for the first three digits of $p(n)$ in base 2.
\begin{table}[ht]
\caption{$B(d,x,2;p(n))$} % title of Table
\scriptsize
\begin{tabular}{c c c c c}  % centered columns (4 columns)
\hline\hline                        %inserts double horizontal lines
$x$ & $d$ = 100 & $d$=101 & $d$=110 &$d$=111\\ [0.5ex] % inserts table
%heading
\hline                    % inserts single horizontal line
200&0.285&0.270&0.205&0.225\\
400&0.308&0.273&0.209&0.205\\
600&0.313&0.267&0.217&0.198\\
800&0.314&0.263&0.219&0.201\\
1000&0.315&0.262&0.220&0.200\\
5000&0.321&0.264&0.222&0.194\\
$\downarrow$ & $\downarrow$ & $\downarrow$ & $\downarrow$ & $\downarrow$\\
$\infty$?&0.322&0.263&0.222&0.192\\

\hline     %inserts single line
\end{tabular}
\label{table:nonlin}  % is used to refer this table in the text
\end{table}%can we shorten this a little in width -- a little far to the right

\indent
To start, we begin with the following definition and theorem.
\begin{dfn}
We say that an integer-valued function $a(n)$ is \emph{good} whenever
\[
a(n)\sim b(n)e^{c(n)}
\]
(where $f(x) \sim g(x)$ means that $\displaystyle \lim_{x \to \infty} f(x)/g(x) =1$)
and the following conditions are satisfied:
\begin{enumerate}
\item{There exists some integer $h\geq 1$ such that $c(n)$ is $h$-differentiable and $c^{(h)}(n)$ tends to zero monotonically for sufficiently large $n$.}
\item{$\displaystyle\lim_{n \to \infty}n|c^{(h)}(n)| = \infty$}
\item{$\displaystyle\lim_{n \to \infty}\frac{D^{(h)}\log b(n)}{c^{(h)}(n)} = 0$, where $D^{(h)}$ denotes the $h^{th}$ derivative.}
\end{enumerate}
\end{dfn}
Our first result is the following theorem. 
\begin{thm}\label{goodthm}
If $a(n)$ is good, then $a(n)\in \mathcal{B}$.
\end{thm}
As a special case of this result we obtain the following corollary. 
\begin{corollary}\label{p(n)}
The partition function $p(n)\in \mathcal{B}.$
\end{corollary}
\begin{proof}
Using the celebrated Hardy-Ramanujan asymptotic
\[ p(n)\sim \frac{1}{4n\sqrt{3}}\cdot e^{\pi\sqrt{2n/3}}
,\]
it immediately follows that $p(n)$ is good and hence Benford.
\end{proof}
Next, we explicitly demonstrate that a large class of arithmetic functions arising from the coefficients of modular forms is Benford, as in the following theorem.
\begin{thm}\label{mf}
Suppose that $M(z)$ is a weakly holomorphic modular form (see $\S 3.1$ for the definition) of weight $\frac{1}{2} \geq k \in \frac{1}{2}\mathbb{Z}$ with integral Fourier coefficients and at least one pole. Then the nonzero coefficients of $M(z)$ are Benford.
\end{thm}
Generalizing Corollary \ref{p(n)}, we obtain the following corollary. 
\begin{corollary}\label{p_s(n)}
For any positive integer we say that a partition is \emph{s-regular} if it has no part divisible by $s$. Denote by $b_s(n)$ the number of $s$-regular partitions. Then Theorem \ref{mf} implies that $b_s(n)$ is Benford.
\end{corollary}
Applying Theorem \ref{goodthm} to generalized Dedekind Eta-products, we have the following corollary.
\begin{corollary}
For $\delta \geq 2$ and $0 < g < \lfloor \frac{\delta +1}{2}\rfloor$ define $r_{g,\delta}(n)$ to be the number of partitions of $n$ into parts congruent to $\pm g$ (mod $\delta$).  This includes the famous Rogers-Ramanujan functions
\[
\sum_{n=0}^{\infty}\frac{q^{n^{2}+an}}{(q;q)_n}=\prod_{n=0}^{\infty}\frac{1}{(1-q^{5n+a+1})(1-q^{5n+4-a})},
\]
for $a=0,1$, where for $n>0$, $(a;q)_n:=\displaystyle \prod_{i=0}^{n-1}(1-aq^i)$ is the $q$-Pochhammer symbol.  Note that these equalities are given by the celebrated Rogers-Ramanujan identities.
Then $r_{g,\delta}(n)$ is Benford.
\end{corollary}
\emph{Remark.}

It is well-known that every nonconstant weakly holomorphic modular form with non-positive weight has a pole. Though this is not true for weights $k\geq \frac{1}{2}$, Theorem \ref{mf} applies to all weakly holomorphic modular forms of weight $\frac{1}{2}$ with a pole and essentially all weakly holomorphic modular forms with higher weight which have a pole. See the discussion and example in $\S 3.2$ for more details. For a more detailed discussion on weakly holomorphic modular forms see $\S 3.1$ and \cite{KO}.

\section{Theorems on Uniform Distribution}
\subsection{Preliminaries on Uniform Distribution and Properties of Benford Spaces}

We first note that proving a sequence is Benford reduces to a problem of uniform distribution. Using this formulation and some classical theorems on uniform distribuiton, we derive a set of sufficient conditions for being Benford.
\begin{dfn}
For a real sequence $a(n)$, $N\in \mathbb{N}$, and $E\subseteq \mathbb{R}$, let $A(E,N,a(n)):=\# \{n\leq N \mid a(n)\in E\}$. Then $a(n)$ is said to be \emph{uniformly distributed mod 1}, if,  for all intervals $[a,b) \subseteq [0,1)$,
\[
\displaystyle\lim_{n\to \infty}\frac{A([a,b),n,\{ a(n) \})}{n} = b-a,
\]
where $<x>$ denotes the fractional part of $x$.
\end{dfn}

We begin by recalling a result of Diaconis that $ a(n) \in \mathcal{B}$ if and only if $\log_{k}(a(n))$ is uniformly distributed mod 1 for all $k$ \cite{PD}. To prove our results we need the following preliminary theorem.
\begin{thm}[Weyl's Criterion]
The sequence $a(n)$ is uniformly distributed mod 1 if and only if 
\[
\displaystyle\lim_{N\to \infty}\frac{1}{N}\sum_{n=1}^{N}e^{2\pi k \cdot a(n)}=0
\]
for any $0\neq k\in \mathbb{Z}.$
\end{thm}

As a corollary to Van der Corput's Difference Theorem we recall the following theorem:
\begin{thm}[Theorem 3.5 of \cite{UD}]
\label{vdk}
Let k $\in \mathbb{N}$ and $f(x)$ be a function defined for $x\geq 1$ which is k-times differentiable for all $x\geq x_{0}$ for some $x_{0}\in \mathbb{R}_{+}$. Suppose that $f^{(k)}$ is eventually monotonic,
\[ 
\displaystyle\lim_{x \to \infty}\ f^{(k)}(x) =0,
\]
and
\[
\displaystyle\lim_{x\to \infty} x|{f^{(k)}(x)}|=\infty.
\]
Then $\{f(n) : n\in \mathbb{N}\}$ is uniformly distributed mod 1.
\end{thm}

Finally, we note a few basic facts regarding Benford spaces.
\\
\begin{lemma}
\text{                }
\begin{enumerate}

	\item{If $f(n)\in\mathcal{B}$ and $c\in\mathbb{R} - \{0\}$, then $cf(n)\in\mathcal{B}$.
		}
	\item{If $f(n)\in\mathcal{B}$ is nonzero, then $\frac{1}{f(n)}\in\mathcal{B}$.
	}
	\item{If $f(n)\in\mathcal{B}$ and $f(n)\sim g(n)$, then $g(n)\in\mathcal{B}$.
	}
\end{enumerate}
\end{lemma}

Note that (1) and (2) follow directly from the definition of uniform distribution and (3) follows easily using Weyl's Criterion.

\subsection{Proof of Theorem \ref{goodthm}}
Here we prove Theorem \ref{goodthm} using these results on uniform distribution.
\begin{proof}
First note that by Part 3 of the lemma, it suffices to show that $b(n)e^{c(n)}$ is Benford. This is equivalent to showing the uniform distribution of $\log(b(n)) + c(n)$ (mod 1). We see that $c(n)$ satisfies the limit conditions of Theorem \ref{vdk} by definition. By the order assumption on $D^{(h)}(\log (b(n)))$, the limits are unaffected upon adding $\log b(n)$.
\end{proof}
\section{Proof of Theorem \ref{mf}}
\subsection{Weakly Holomorphic Modular Forms}
Here we give some of the preliminary definitions and theorems on weakly holomorphic modular forms. For a more detailed reference, see \cite{KO}. Throughout this discussion let $q := e^{2\pi i z}$. 
\\
\indent
For any positive integer $N$, we define the \emph{level $N$ congruence subgroups of $\emph{SL}_2(\mathbb{Z})$} by
\\
$\Gamma_0(N):=\Big{\{} \Big(
\begin{matrix}
a&b\\
c&d
\end{matrix}
\Big)
\in \text{SL}_2(\mathbb{Z}) : c\equiv 0$ mod N$\Big{\}}$
\\
$\Gamma_1(N):=\Big{\{} \Big(
\begin{matrix}
a&b\\
c&d
\end{matrix}
\Big)
\in \text{SL}_2(\mathbb{Z}) : a \equiv d \equiv 1$ mod $N$ and $ c\equiv 0$ mod N$\Big{\}}$
\\
$\Gamma(N):=\Big{\{} \Big(
\begin{matrix}
a&b\\
c&d
\end{matrix}
\Big)
\in \text{SL}_2(\mathbb{Z}) : a \equiv d \equiv 1$ mod $N$ and $ b\equiv c \equiv 0$ mod N$\Big{\}}$

Suppose that $\Gamma$ is a congruence subgroup of $\text{SL}_2(\mathbb{Z})$. A \emph{cusp} is an equivalence class of $\mathbb{P}^1(\mathbb{Q})$ under the usual fractional linear action of $\Gamma$.
\\

Let $f(z)$ be a meromorphic function on the upper half plane $\mathbb{H}$, $k\in \mathbb{Z}$, and $\Gamma$ be a congruence subgroup of level $N$. First, we define the ``slash" operator by
\[
(f|_k \gamma)(z) := (\det{\gamma})^{k/2}(cz+d)^{-k}f(\gamma z),
\]
where $\gamma\in\text{SL}_2(\mathbb{Z})$ and $\gamma z := \frac{az+b}{cz+d}$.

Then $f(z)$ is said to be a \emph{meromorphic modular form with integer weight k on} $\Gamma$ if the following hold:
\begin{enumerate}
\item{
For all $z\in\mathbb{H}$ and all $ \Big( \begin{matrix}
a&b\\
c&d
\end{matrix} \Big) \in \Gamma$ we have that 
\[
f\Big(\frac{az+b}{cz+d}\Big) = (cz+d)^kf(z).
\]
}
\item{
If $\gamma_0\in\text{SL}_2(\mathbb{Z})$, then  $(f|_k \gamma_0)(z)$ has a Fourier expansion of the form
\[
(f|_k \gamma_0)(z) = \displaystyle\sum_{n\geq n_{\gamma_0}}a_{\gamma_0}(n)q^{n/N}.
\]
}
\end{enumerate}
If $k=0$, then $f(z)$ is known as a \emph{modular function on} $\Gamma$.

 We say that $f(z)$ is a \emph{weakly holomorphic modular form} if its poles are supported at the cusps of $\Gamma$. 
\\
\indent
Here we recall the notion of a modular form of half-integral weight. First we define $\Big( \frac{c}{d}\Big)$ and $\epsilon_d$. If $d$ is an odd prime, let $\Big( \frac{c}{d}\Big)$ be the usual Legendre symbol. For positive odd $d$, define $\Big( \frac{c}{d}\Big)$ by multiplicativity. For negative odd $d$, we let
\[
\Big( \frac{c}{d}\Big):=
\begin{cases}
\Big( \frac{c}{|d|}\Big) & \text{if } d<0 \text{ and } c>0,\\
-\Big( \frac{c}{|d|}\Big) & \text{if } d<0 \text{ and } c<0.
\end{cases}
\]
Also let $\Big( \frac{0}{\pm 1}\Big)=1$. Define $\epsilon_d$, for odd $d$, by
\[
\epsilon_d :=
\begin{cases}
1 & \text{if } d\equiv 1 \text{ mod } 4,\\
\emph{i} & \text{if } d \equiv 3 \text{ mod } 4.
\end{cases}
\]

Let $\lambda$ be a nonnegative integer and $N$ a positive integer. Furthermore, suppose that $\chi$ is a Dirichlet character modulo $4N$. A meromorphic function $g(z)$ on $\mathbb{H}$ is said to be a \emph{modular form with Nebentypus $\chi$ and weight $\lambda + \frac{1}{2}$} if it is meromorphic at the cusps of $\Gamma$, and if 
\[
g\Big(\frac{az+b}{cz+d}\Big) = \chi(d) \Big( \frac{c}{d}\Big)^{2\lambda + 1} \epsilon_d^{-1-2\lambda} (cz+d)^{\lambda + \frac{1}{2}} g(z)
\]
for all $\Big(
\begin{matrix}
a&b\\
c&d
\end{matrix}
\Big)
\in \Gamma_0(4N).$

\noindent\emph{Remark.} As in the integral case, we say that a meromorphic modular form $M(z)$ is \emph{weakly holomorphic} if its poles are supported on the cusps (that is, $M$ is holomorphic on $\mathbb{H}$).

For example, consider Dedekind's eta-function,
\[
\eta(z) := q^{1/24}\displaystyle\prod_{n=1}^{\infty}(1-q^n),
\]
which is a non-vanishing, weakly holomorphic modular form with weight 1/2. The following theorem becomes useful in generating more examples of modular forms (see for example, \cite{KO}).
\begin{thm}
Let $f(z) = \prod_{\delta|N}\eta(\delta z)^{r_\delta}$ be an eta-quotient with $k=\frac{1}{2}\sum_{\delta|N}r_{\delta}\in\mathbb{Z}$. Then if 
\[
\displaystyle\sum_{\delta|N}\delta r_{\delta} \equiv 0 \emph{ (mod 24)}
\]
and
\[
\displaystyle\sum_{\delta|N}\frac{N}{\delta} r_{\delta} \equiv 0 \emph{ (mod 24)},
\]
then $f(z)$ satisfies
\[
\Big(\frac{az+b}{cz+d}\Big) = \chi(d)(cz+d)^kf(z)
\]
for every $\Big(
\begin{matrix}
a&b\\
c&d
\end{matrix}
\Big) \in \Gamma_0(N)$. Here the character $\chi$ is defined by $\chi(d) := \big(\frac{(-1)^{k}s)}{d}\big)$, where $s := \prod_{\delta|N}\delta^{r_{\delta}}$.
\end{thm}

\subsection{Proof of Theorem \ref{mf}}
We are now in position to prove Theorem \ref{mf}.
\begin{proof}
We first consider the case where $k$ is nonpositive. By hypothesis, $M$ has at least one pole at a cusp of the relevant subgroup $\Gamma \subseteq \text{SL}_{2}(\mathbb{Z})$. Recall that a harmonic Maass form is a function defined in a similar manner as a modular form, but with relaxed growth conditions and the additional requirement that it lies in the kernel of a certain weight $k$ Laplacian operator. Every such harmonic Maass form can be decomposed as a sum of a holomorphic part $f^{(+)}$ and a nonholomorphic part $f^{(-)}$. It is well-known that every weakly holomorphic modular form is a harmonic Maass form. In addition,the nonholomorphic part $f^{(-)}$ is identically zero. By a recent paper of Bringmann and Ono \cite{KK2}, we have exact formulas for the holomorphic part of harmonic Maass forms for $k\leq \frac{1}{2}$. Because in our setting $f^{(-)}$ is identically zero, these give us exact formulas for the Fourier coefficients of $M$. The explicit formulas are quite complicated, so for brevity note that if we write $M:=\sum_{n\geq n_{0}}m(n)q^{n}$, it is easy to show that there exists an integer $t\geq 1$ such that for all $0\leq r \leq t$ the nonzero coefficients satisfy
\[
m(tn+r) \sim K(M,r,t) \cdot (tn+r)^{\frac{k-1}{2}} \cdot I_{1-k}(\alpha (M,r,t) \sqrt{tn+r})
\]
where $K(M,r,t)$ and $\alpha (M,r,t)$ are constants and $I$ denotes the modified Bessel function of the first kind. Using the standard asymptotics for Bessel functions, namely that
\[
I_{\alpha}(x) \sim \frac{e^x}{\sqrt{2\pi x}}\cdot \Big(1+\frac{(1-2\alpha)(1+2\alpha)}{8x}+\dots\Big),
\]
 we see that the sequence of nonzero coefficients is good and hence Benford. 
\\
\indent If $k=1/2$, then by Theorem 1.2 in \cite{KK2}, we have a new exact formula of the same form with the addition of a finite number of terms which do not affect the asymptotic. 
\end{proof}
\noindent\emph{Remark.} Although we have stated and proved Theorem 1.3 for weights $k \leq \frac{1}{2}$, it turns out that a suitably modified version holds for all weights.  Let $M(z)$ be a weakly holomorphic modular form with weight $k > \frac{1}{2}$  with integral coefficients and at least one pole, which by definition must be at a cusp.  We can decompose $M(z) = P(z)+H(z),$ where $P(z)$ is a linear combination of Maass-Poincar\'{e} series and $H(z)$ is a holomorphic modular form.  Note that by \cite{KK2} this decomposition is stated only for the case $1/2 \geq k\in\frac{1}{2}\mathbb{Z},$ but in fact it will hold for all $k\in\frac{1}{2}\mathbb{Z}.$  When $k\leq 0,$ it follows that $H(z)=0,$ and when $k = 1/2,$ $H(z)$ is a linear combination of readily understood theta functions by the Serre-Stark basis theorem \cite{KK2}.  Therefore, when $k\leq \frac{1}{2}$ this decomposition gives exact formulas for the coefficients.  For other $k,$ we do not have a good theory for the possible $H(z),$ therefore we typically cannot obtain exact formulas for the coefficients of $M(z)$.  However, it is well-known that the coefficients of all such $H(z)$ are bounded by a fixed power of $n$. To see this, note that the space of holomorphic modular forms is spanned by Eisenstein series and cusp forms. From the expansion of Eisenstein series in terms of divisor functions, it is easily seen that the coefficients are bounded by a polynomial. Cusp forms are also bounded by a power of $n$ by Proposition 8 on page 23 of \cite{1-2-3}. Consequently, we obtain sufficient asymptotics for all coefficients which are nonvanishing in $P(z).$  These asymptotics will have subexponential growth, which will dominate the coefficients from $H(z)$.\\\\
\emph{Example.}
To illustrate the above remark, consider the weight 12 weakly holomorphic modular form $M(z):=j(z)\cdot E_4(z)^3$, where $j(z)$ is the $j$-modular invariant and $E_4(z)$ is the Eisenstein series of order 4.  The first few terms are 
\[
M(z) = q^{-1}+1464 +911844q+313589120q^2+\cdots
\]
This function has a pole at $\infty$ so the above remark applies as is illustrated in the base 3 case below.
\begin{table}[ht]
\caption{$B(d,x,3;j(z)\cdot E_4(z)^3$)} % title of Table
\scriptsize
\begin{tabular}{c c c c c c c }  % centered columns (4 columns)
\hline\hline                        %inserts double horizontal lines
$x$ & $d$ = 10 & $d$=11 & $d$=12 &$d$=20 &$d$=21 &$d$=22\\ [0.5ex] % inserts table
%heading
\hline                    % inserts single horizontal line
500&.2440&.2020&.1700&.1320&.1300&.1220\\
1000&.2590&.2050&.1610&.1360&.1270&.1120\\
1500&.2627&.2027&.1633&.1373&.1273&.1067\\
2000&.2635&.2030&.1660&.1375&.1245&.1055\\
$\downarrow$ & $\downarrow$ & $\downarrow$ & $\downarrow$ & $\downarrow$ & $\downarrow$& $\downarrow$\\
$\infty$?&.2619&.2031&.1660&.1403&.1215&.1072\\

\hline     %inserts single line
\end{tabular}
\label{table:nonlin}  % is used to refer this table in the text
\end{table}%can we shorten this a little in width -- a little far to the right

\section{Proof of Corollaries 1.4 and 1.5}
In this section we use Theorem 1.1 to prove new examples of Benford sequences in interesting cases.
\subsection{Proof of Corollary 1.4}
First note that 
\[
\sum_{n=0}^{\infty}b_s(n)q^n=\prod_{n=1}^{\infty} \frac{1-q^{sn}}{1-q^n}.
\]
Then in terms of Dedekind's eta function, $\eta (z),$ we have that
\[
\sum_{n=0}^{\infty}b_{s}(n)q^{24n+s-1}=\frac{\eta (24sz)}{\eta (24z)}.
\]
\indent Then by Theorem 1.64 \cite{KO}, we see that $b_{s}(n)$ is obtained as the coeffients of a nonconstant modular form of weight zero on $\Gamma_0(576s)$. Moreover, as $\eta (z)$ is nonvanishing on $\mathbb{H}$, the given modular form is weakly holomorphic and hence $b_{s}(n)$ is Benford.
\subsection{Proof of Corollary 1.5}
It is easy to see that
\[
\sum_{n=0}^{\infty}r_{g,\delta}(n)q^{n}=\prod_{1\leq n \equiv g \text{ (mod } \delta)} \frac{1}{(1-q^{n})}\prod_{1\leq n \equiv -g \text{ (mod } \delta)} \frac{1}{(1-q^{n})}.
\]
\indent
For example, we have the Rogers-Ramanujan functions
\[
\sum_{n=0}^{\infty}\frac{q^{n^{2}+an}}{(q;q)_n}=\prod_{n=0}^{\infty}\frac{1}{(1-q^{5n+a+1})(1-q^{5n+4-a})},
\]
for $a=0,1$.
\\
\indent
Our claim is that for all $g$ and $\delta$ satisfying the above restrictions $r_{g,\delta}$ is Benford. To see this, define the \emph{generalized Dedekind Eta-product} $\eta_{g,\delta}(z)$ by 
\[
\eta_{g,\delta}(z):=e^{2\pi i P_{2}(\frac{g}{\delta})\delta z}\prod_{1\leq n \equiv g \text{ (mod } \delta)} (1-q^{n})\prod_{1\leq n \equiv -g \text{ (mod } \delta)}(1-q^{n}).
\]
\indent
Here $P_2(t)$ is the second Bernoulli function $P_2(t):=\{ t\}^2 - \{ t\}+\frac{1}{6}$, and $\{ t\}$ denotes the fractional part of $t$. Then by \cite{GEP}, for $g\neq0,\frac{1}{2}\delta$, it follows that $\eta_{g,\delta}(z)$ is a modular form of weight zero which is nonvanishing on $\mathbb{H}$. The generating functions for $r_{g,\delta}$ are then given by the Fourier expansion of $\frac{1}{\eta_{g,\delta}(z)}$ shifted by an integral power of $q$. As this is a nonconstant weakly holomorphic modular form of weight zero, it follows that $r_{g,\delta}(n)$ is Benford.

\end{document}